\documentclass[11pt,reqno]{amsart}
\usepackage{color}

\usepackage{amsmath}
\usepackage{amsfonts,amscd}
\usepackage{amssymb}
\usepackage{url}
\usepackage{hyperref}

\usepackage[english]{babel}
\usepackage{booktabs}
\usepackage{tikz}

\theoremstyle{plain}
\newtheorem{theorem}                 {Theorem}      [section]

\newtheorem{lemma}        [theorem]  {Lemma}
\newtheorem{proposition}  [theorem]  {Proposition}

\theoremstyle{definition}

\newtheorem{example}      [theorem]  {Example}

\newtheorem{definition}   [theorem]  {Definition}

\newtheorem{remark}       [theorem]  {Remark}

\numberwithin{equation}{section}

\def \rn{{\mathbb R}}

\def \C{{\mathbb C}}
\def \cn{{\mathbb C}}

\def \B{\mathcal B}

\def \E{\mathcal E}
\def \F{\mathcal F}
\def \H{\mathcal H}

\def \nab#1#2{\hbox{$\nabla$\kern -.3em\lower 1.0 ex
    \hbox{$#1$}\kern -.1 em {$#2$}}}

\DeclareMathOperator{\Ad}{Ad}

\def \GLR#1{\text{\bf GL}_{#1}(\rn)}

\def \SL2{\widetilde{\text{\bf SL}}_{2}(\rn)}

\def \SO#1{\text{\bf SO}(#1)}

\def \U#1{\text{\bf U}(#1)}

\def \SU#1{\text{\bf SU}(#1)}

\def \Sp#1{\text{\bf Sp}(#1)}

\DeclareMathOperator{\Div}{div}

\DeclareMathOperator{\trace}{trace}
\DeclareMathOperator{\rank}{rank}

\numberwithin{equation}{section}

\allowdisplaybreaks

\begin{document}

\subjclass[2020]{53C35, 53C43, 58E20}

\keywords{p-harmonic functions, symmetric spaces}

\author{Elsa Ghandour}
\address{Mathematics, Faculty of Science\\
	University of Lund\\
	Box 118, Lund 221 00\\
	Sweden}
\email{Elsa.Ghandour@hotmail.com}

\author{Sigmundur Gudmundsson}
\address{Mathematics, Faculty of Science\\
	University of Lund\\
	Box 118, Lund 221 00\\
	Sweden}
\email{Sigmundur.Gudmundsson@math.lu.se}

\title
[$p$-Harmonic functions on the real Grassmannians]
{Explicit $p$-harmonic functions \\ on  the real Grassmannians}

\begin{abstract}
In this work we use the method of eigenfamilies to construct explicit complex-valued proper  $p$-harmonic functions on the compact real Grassmannians. We also find proper $p$-harmonic functions on the real flag manifolds which do not descend onto any of the real Grassmannians.
\end{abstract}


\maketitle

\section{Introduction}

Mathematicians and physicists have been studying biharmonic functions for nearly two centuries. Applications have been found within physics for example in continuum mechanics, elasticity theory, as well as two-dimensional hydrodynamics problems involving Stokes flows of incompressible Newtonian fluids. Until just a few years ago, with only very few exceptions, the domains of all known explicit proper $p$-harmonic functions
have been either surfaces or open subsets of flat Euclidean space. 
A recent development has changed this situation and can be traced at the regularly updated online bibliography \cite{Gud-p-bib}, maintained by the second author. 
\smallskip

A natural habitat for the study of complex-valued $p$-harmonic functions $\phi:(M,g)\to\cn$, on Riemannian manifolds, is found by assuming that the domain is a symmetric space.  These were classified by the pioneering work of \'Elie Cartan in the late 1920s. For this we refer to the standard work \cite{Hel} by Helgason.  The irreducible Riemannian symmetric spaces come in pairs each consisting of a  compact space $U/K$ and its non-compact dual $G/K$.  In their recent article \cite{Gud-Sob-1}, the authors construct the first known explicit proper $p$-harmonic functions ($p\ge 2$) for the compact cases 
$$\SO n,\ \ \SU n,\ \ \Sp n$$ 
of Type II, see also \cite{Gud-Mon-Rat-1} and \cite{Gud-Sif-1} for the special case when $p=2$.  For compact symmetric spaces of Type I, the authors of \cite{Gud-Sif-Sob-2} deal with the cases of 
$$\SU n/\SO n,\ \ \Sp n/\U n,\ \ \SO {2n}/\U n,\ \ \SU{2n}/\Sp n.$$ 

For complex-valued $p$-harmonic functions on symmetric spaces we have a duality principle first introduced for harmonic morphisms in \cite{Gud-Sve-1} and later developed for $p$-harmonic functions in \cite{Gud-Mon-Rat-1}.  This means that a solution on the compact $U/K$ induces, in a natural way, a solution on its non-compact dual $G/K$ and vice versa.  For this reason we here only discuss the compact cases. 
\smallskip

It is the principal aim of this work to construct the first known explicit complex-valued proper $p$-harmonic functions on the real Grassmannians $\SO{m+n}/\SO m\times\SO n$.  Our method is inspired by the classical spherical harmonics on the standard round sphere $S^n=\SO{n+1}/\SO n$, see Remark \ref{remark-spherical-harmonics}.

\section{Proper $p$-Harmonic Functions}
\label{section-p-harmonic-functions}

In this section we describe a method for manufacturing  complex-valued proper $p$-harmonic functions on Riemannian manifolds. This was recently introduced in \cite{Gud-Sob-1}.
\medskip

Let $(M,g)$ be an $m$-dimensional Riemannian manifold and $T^{\cn}M$ be the complexification of the tangent bundle $TM$ of $M$. We extend the metric $g$ to a complex-bilinear form on $T^{\cn}M$.  Then the gradient $\nabla\phi$ of a complex-valued function $\phi:(M,g)\to\cn$ is a section of $T^{\cn}M$.  In this situation, the well-known complex linear {\it Laplace-Beltrami operator} (alt. {\it tension field}) $\tau$ on $(M,g)$ acts locally on $\phi$ as follows
$$
\tau(\phi)=\Div (\nabla \phi)=\sum_{i,j=1}^m\frac{1}{\sqrt{|g|}} \frac{\partial}{\partial x_j}
\left(g^{ij}\, \sqrt{|g|}\, \frac{\partial \phi}{\partial x_i}\right).
$$
For two complex-valued functions $\phi,\psi:(M,g)\to\cn$ we have the following well-known relation
\begin{equation}\label{equation-fundamental}
	\tau(\phi\cdot \psi)=\tau(\phi)\cdot\phi +2\cdot\kappa(\phi,\psi)+\phi\cdot\tau(\psi),
\end{equation}
where the complex bilinear {\it conformality operator} $\kappa$ is given by $$\kappa(\phi,\psi)=g(\nabla \phi,\nabla \psi).$$  Locally this satisfies
$$\kappa(\phi,\psi)=\sum_{i,j=1}^mg^{ij}\cdot\frac{\partial\phi}{\partial x_i}\frac{\partial \psi}{\partial x_j}.$$

\begin{definition}\cite{Gud-Sak-1}\label{definition-eigenfamily}
Let $(M,g)$ be a Riemennian manifold. Then a complex-valued function $\phi:M\to\cn$ is said to be an {\it eigenfunction} if it is eigen both with respect to the Laplace-Beltrami operator $\tau$ and the conformality operator $\kappa$ i.e. there exist complex numbers $\lambda,\mu\in\cn$ such that $$\tau(\phi)=\lambda\cdot\phi\ \ \text{and}\ \ \kappa(\phi,\phi)=\mu\cdot \phi^2.$$
A set $\E =\{\phi_i:M\to\cn\ |\ i\in I\}$ of complex-valued functions is said to be an {\it eigenfamily} on $M$ if there exist complex numbers $\lambda,\mu\in\cn$ such that for all $\phi,\psi\in\E$ we have 
$$\tau(\phi)=\lambda\cdot\phi\ \ \text{and}\ \ \kappa(\phi,\psi)=\mu\cdot \phi\cdot\psi.$$ 
\end{definition}

In this work we are mainly interested in complex-valued proper $p$-harmonic functions.  They are defined as follows.

\begin{definition}\label{definition-proper-p-harmonic}
	Let $(M,g)$ be a Riemannian manifold. For a positive integer $p$, the iterated Laplace-Beltrami operator $\tau^p$ is given by
	$$\tau^{0} (\phi)=\phi\ \ \text{and}\ \ \tau^p (\phi)=\tau(\tau^{(p-1)}(\phi)).$$	We say that a complex-valued function $\phi:(M,g)\to\cn$ is
	\begin{enumerate}
		\item[(i)] {\it $p$-harmonic} if $\tau^p (\phi)=0$, and
		\item[(ii)] {\it proper $p$-harmonic} if $\tau^p (\phi)=0$ and $\tau^{(p-1)}(\phi)$ does not vanish identically.
	\end{enumerate}
\end{definition}

Our construction of complex-valued proper $p$-harmonic functions, on the real Grassmannians, is based on the following method, recently introduced in \cite{Gud-Sob-1}.

\begin{theorem}\label{theorem-p-harmonic}
	Let $\phi:(M,g)\to\cn$ be a complex-valued function on a Riemannian manifold and $(\lambda,\mu)\in\cn^2\setminus\{0\}$ be such that the tension field $\tau$ and the conformality operator $\kappa$ satisfy 
	$$\tau(\phi)=\lambda\cdot \phi\ \ \text{and}\ \ \kappa(\phi,\phi)=\mu\cdot\phi^2.$$
	Then for any positive natural number $p$, the non-vanishing function 
	$$\Phi_p:W=\{ x\in M \,\mid\, \phi(x) \not\in (-\infty,0] \}\to\cn$$ with 
	$$\Phi_p(x)= 
	\begin{cases}
		c_1\cdot\log(\phi(x))^{p-1}, 							& \text{if }\; \mu = 0, \; \lambda \not= 0\\[0.2cm]	c_1\cdot\log(\phi(x))^{2p-1}+ c_{2}\cdot\log(\phi(x))^{2p-2}, 								& \text{if }\; \mu \not= 0, \; \lambda = \mu\\[0.2cm]
		c_{1}\cdot\phi(x)^{1-\frac\lambda{\mu}}\log(\phi(x))^{p-1} + c_{2}\cdot\log(\phi(x))^{p-1},	& \text{if }\; \mu \not= 0, \; \lambda \not= \mu
	\end{cases}
	$$ 
	is a proper $p$-harmonic function.  Here $c_1,c_2$ are complex coefficients, not both zero.
\end{theorem}

\section{Lifting Properties}

We shall now present an interesting connection between the theory of complex-valued $p$-harmonic functions and the notion of harmonic morphisms. Readers not familiar with harmonic morphisms are advised to consult the standard text \cite{Bai-Woo-book} and the regularly updated online bibliography \cite{Gud-bib}.

\begin{proposition}\label{proposition-lift}
	Let $\pi:(\hat M,\hat g)\to (M,g)$ be a submersive harmonic morphism
	between Riemannian manifolds. Further let $f:(M,g)\to\C$ be a smooth function and
	$\hat f:(\hat M,\hat g)\to\C$ be the composition $\hat f=f\circ\pi$.
	If $\lambda:\hat M\to\rn^+$ is the dilation of $\pi$ then the tension
	fields $\tau$ and $\hat\tau$  satisfy
	$$
	\tau(f)\circ\pi=\lambda^{-2}\hat\tau(\hat f)\ \ \text{and}\ \ \tau^p(f)\circ\pi=\lambda^{-2}\hat\tau(\lambda^{-2}\hat\tau^{(p-1)}(\hat f))
	$$
	for all positive integers $p\ge 2$.
\end{proposition}

\begin{proof}
	The harmonic morphism $\pi$ is a horizontally conformal, harmonic map. Hence the well-known composition law for the tension field gives
	\begin{eqnarray*}
		\hat\tau(\hat f)
		&=&\hat\tau(f\circ\pi)\\
		&=&\text{trace}\nabla df(d\pi,d\pi)+ df(\hat\tau(\pi))\\
		&=&\lambda^2\tau(f)\circ\pi+df(\hat\tau(\pi))\\
		&=&\lambda^2\tau(f)\circ\pi.
	\end{eqnarray*}
	For the second statement, set $h=\tau(f)$ and
	$\hat h=\lambda^{-2}\cdot\hat\tau(\hat f)$.
	Then $\hat h=h\circ\pi$ and it follows from the first step that
	$$\hat\tau(\lambda^{-2}\hat\tau(\hat f))=\hat\tau(\hat h)
	=\lambda^2\tau(h)\circ\pi=\lambda^2\tau^2(f)\circ\pi,$$
	or equivalently,
	$$\tau^2(f)\circ\pi=\lambda^{-2}\hat\tau(\lambda^{-2}\tau(\hat f)).$$
	The rest follows by induction.
\end{proof}

Let $G$ be the special orthogonal group $\SO{m+n}$, with subgroup $K=\SO m\times\SO n$.  Then the standard biinvariant Riemannian metric on $G$, induced by the Killing form, is $\Ad(K)$-invariant and induces a Riemannian metric on the symmetric quotient space $G/K$.  Moreover, the natural projection $\pi: G\to G/K$ is a Riemannian submersion with totally geodesic fibres, hence a harmonic morphism satisfying the conditions in Proposition \ref{proposition-lift}

\section{Eigenfunctions on the Real Grassmannians}
\label{section-real}

The special orthogonal group $\SO{m+n}$ is the compact subgroup of the real general linear group $\GLR{m+n}$ of invertible matrices satisfying 
$$\SO{m+n}=\{x\in\GLR{m+n}\,|\, x\cdot x^t=I\ \text{and}\ \det x=1\}.$$ Its standard representation on $\cn^{m+n}$ is denoted by 
$$\pi: x\mapsto 
\begin{bmatrix}
x_{11} & \cdots &x_{1,m+n}\\
\vdots & \ddots &\vdots\\
x_{m+n,1} & \cdots &x_{m+n,m+n}
\end{bmatrix}.
$$
For this situation we have the following basic result from Lemma  4.1 of \cite{Gud-Sak-1}.

\begin{lemma}\label{lemma-fundamental-SOn}
For $1\le j,k,\alpha,\beta\le m+n$, let $x_{j\alpha}:\SO {m+n}\to\rn$ be the real-valued matrix elements of the standard representation of $\SO {m+n}$.  Then the following relations hold  $$\hat \tau(x_{j\alpha})=-\,\frac {(m+n-1)}2\cdot x_{j\alpha},$$
$$\hat\kappa(x_{j\alpha},x_{k\beta})=-\,\frac 12\cdot (x_{j\beta}x_{k\alpha}-\delta_{jk}\delta_{\alpha\beta}).$$
\end{lemma}

For $1\le j,\alpha\le m+n$, we now define the real-valued functions $\hat \phi_{j\alpha}:\SO{m+n}\to\rn$ on the special orthogonal group $\SO{m+n}$ by
$$\hat\phi_{j\alpha}(x)=\sum_{t=1}^{m}x_{jt}\, x_{\alpha t}.$$
These functions are $\SO{m}\times\SO{n}$-invariant and hence induce functions on the compact quotient space $\SO{m+n}/\SO{m}\times\SO{n}$ i.e. the real Grassmannian $G_m(\rn^{m+n})$ of $m$-dimensional oriented subspaces of $\rn^{m+n}$.

\begin{lemma}\label{lemma-phi-phi-real}
The tension field $\hat\tau$ and the conformality operator $\hat\kappa$ on the special orthogonal group $\SO{m+n}$ satisfy
\begin{eqnarray*}
\hat\tau(\hat\phi_{j\alpha})&=&-\,(m+n)\cdot\hat\phi_{j\alpha}+\delta_{j\alpha}\cdot m\\
\hat\kappa(\hat\phi_{j\alpha},\hat\phi_{k\beta})
&=&-\,(\hat\phi_{j\beta}\cdot\hat\phi_{k\alpha}+\hat\phi_{jk}\cdot\hat\phi_{\alpha\beta})\\
& &\quad +\,\frac12
\big(
 \delta_{jk}\,\hat\phi_{\alpha\beta}+\delta_{\alpha\beta}\,\hat\phi_{jk}
+\delta_{j\beta}\,\hat\phi_{k\alpha}+\delta_{k\alpha}\,\hat\phi_{j\beta}
\big).
\end{eqnarray*}
\end{lemma}

\begin{proof} The following calculations are based on the equation (\ref{equation-fundamental}) and the formulae in Lemma \ref{lemma-fundamental-SOn}.  For the tension field $\hat\tau$ we have
\begin{eqnarray*}
\hat\tau(\hat\phi_{j\alpha})&=&\sum_{t=1}^{m}\big\{\hat\tau(x_{jt})\cdot x_{\alpha t}+2\cdot\hat\kappa(x_{jt},x_{\alpha t})+x_{jt}\cdot\hat\tau(x_{\alpha t})\big\}\\
&=&-(m+n-1)\cdot\sum_{t=1}^{m} x_{jt}\,  x_{\alpha t}-\sum_{t=1}^{m} (x_{jt}\, x_{\alpha t}-\delta_{j\alpha})\\
&=&-(m+n)\cdot\hat\phi_{j\alpha}+\delta_{j\alpha}\cdot m.
\end{eqnarray*}
For the conformality operator $\kappa$ we then obtain
\begin{eqnarray*}
&&\hat\kappa(\hat\phi_{j\alpha},\hat\phi_{k\beta})\\
&=&\sum_{s,t=1}^{m}\hat\kappa(x_{js} x_{\alpha s},x_{kt}x_{\beta t})\\
&=&\sum_{s,t=1}^{m}\big\{ x_{js} x_{kt}\cdot \hat\kappa(x_{\alpha  s},x_{\beta t})+x_{js}x_{\beta t}\cdot \hat\kappa(x_{\alpha s},x_{k t})\\
&&\qquad\qquad +\,x_{\alpha s} x_{k t}\cdot\hat\kappa(x_{js},x_{\beta t})+x_{\alpha s} x_{\beta t}\cdot \hat\kappa(x_{js},x_{kt})\big\}\\
&=&-\,\frac{1}{2}\sum_{s,t=1}^{m}\big\{ x_{js}x_{kt}\cdot (x_{\alpha t}x_{\beta s}-\delta_{\alpha \beta}\delta_{st})+x_{js}x_{\beta t}\cdot (x_{\alpha t}x_{ks}-\delta_{\alpha k}\delta_{st})\\
& &\qquad\qquad 
+x_{\alpha s}x_{kt}\cdot (x_{jt}x_{\beta s}-\delta_{j\beta}\delta_{st})
+x_{\alpha s}x_{\beta t}\cdot (x_{jt}x_{ks}-\delta_{jk}\delta_{st})\big\}\\
&=&-\,\big (\hat\phi_{j\beta}\cdot\hat\phi_{k\alpha}+\hat\phi_{jk}\cdot\hat\phi_{\beta\alpha}\big )\\
& &\qquad\qquad
+\frac 12\cdot 
\big(
\delta_{jk}\,\hat\phi_{\alpha\beta}+\delta_{\alpha\beta}\,\hat\phi_{jk}
+\delta_{j\beta}\,\hat\phi_{k\alpha}+\delta_{k\alpha}\,\hat\phi_{j\beta}
\big).
\end{eqnarray*}
\end{proof}

\begin{theorem}\label{theorem-real}
Let $A$ be a complex symmetric $(m+n)\times(m+n)$-matrix such that $A^2=0$.  Then the 
$\SO m\times\SO n$-invariant function $\hat\Phi_A:\SO{m+n}\to\cn$, given by
$$\hat\Phi_A(x)=\sum_{j,\alpha=1}^{m+n}a_{j\alpha}\cdot\hat\phi_{j\alpha}(x),$$ 
induces an eigenfunction $\Phi_A:\SO{m+n}/\SO m\times\SO n\to\cn$ on the real Grassmannian with 
$$\tau(\Phi_A)=-\,(m+n)\cdot\Phi_A\ \ \text{and}\ \ \kappa(\Phi_A,\Phi_A)=-\,2\cdot\Phi_A^2$$
if $\text{\rm rank}\,A=1$ and $\trace\,A=0$.
\end{theorem}

\begin{proof}
It is an immediate consequence of Lemma \ref{lemma-phi-phi-real} and the fact that $A$ is traceless that the tension field $\hat\tau$ satisfies $$\hat\tau(\hat\Phi_A)=-\,(m+n)\cdot\hat\Phi_A+m\cdot\trace A=-\,(m+n)\cdot\hat\Phi_A.$$
For the conformality operator $\hat\kappa$ we have 
\begin{eqnarray*}
& &2\cdot\hat\Phi_A^2+\hat\kappa(\hat\Phi_A,\hat\Phi_A)\\
&=&2\cdot\hat\Phi_A^2+\sum_{j,\alpha,k,\beta=1}^{m+n}
\hat\kappa(a_{j\alpha}\hat\phi_{j\alpha},a_{k\beta}\hat\phi_{k\beta})\\
&=&2\cdot\hat\Phi_A^2
+\sum_{j,\alpha,k,\beta=1}^{m+n}a_{j\alpha}a_{k\beta}\cdot 
\hat\kappa(\hat\phi_{j\alpha},\hat\phi_{k\beta})\\
&=&2\cdot\hat\Phi_A^2
-\sum_{j,\alpha,k,\beta=1}^{m+n}a_{j\alpha}a_{k\beta}\,\big \{\hat\phi_{j\beta}\cdot\hat\phi_{k\alpha}+\hat\phi_{jk}\cdot\hat\phi_{\beta\alpha}\big\}\\
& &\qquad 
+\,\frac 12\sum_{j,\alpha,k,\beta=1}^{m+n}a_{j\alpha}a_{k\beta}\,
\big\{\delta_{jk}\,\hat\phi_{\alpha\beta}+\delta_{\alpha\beta}\,\hat\phi_{jk}
+\delta_{j\beta}\,\hat\phi_{k\alpha}+\delta_{k\alpha}\,\hat\phi_{j\beta}\big\}\\
&=&2\cdot\hat\Phi_A^2-\sum_{j,\alpha,k\beta=1}^{m+n}
\big\{a_{j\beta}a_{k\alpha}+a_{jk}a_{\alpha\beta}\big\}\cdot\hat\phi_{j\alpha}\,\hat\phi_{k\beta}\\
& &
+\frac12\sum_{j,\alpha,k,\beta=1}^{m+n}a_{j\alpha}a_{k\beta}
\big(\delta_{ks}\,\hat\phi_{jr}+\delta_{kr}\,\hat\phi_{js}+\delta_{js}\,\hat\phi_{kr}+\delta_{jr}\,\hat\phi_{ks}\big)\\
&=&\sum_{j,\alpha,k,\beta=1}^{m+n}
\big\{2\,a_{j\alpha}a_{k\beta}-a_{j\beta}a_{k\alpha}-a_{jk}a_{\alpha\beta}\big\}\cdot\hat\phi_{j\alpha}\,\hat\phi_{k\beta}+2\sum_{j,\alpha,t=1}^{m+n}a_{jt}a_{\alpha t}\,\hat\phi_{j\alpha}\\
&=&\sum_{j,\alpha,k,\beta=1}^{m+n}
\Big\{\det
\begin{bmatrix}
a_{j\alpha} & a_{j\beta}\\
a_{k\alpha} & a_{k\beta}	
\end{bmatrix}
+
\det
\begin{bmatrix}
a_{j\alpha} & a_{jk}\\
a_{\beta\alpha} & a_{\beta k}	
\end{bmatrix}
\Big\}\cdot\hat\phi_{j\alpha}\,\hat\phi_{k\beta}\\
& &\qquad\qquad\qquad\qquad\qquad\qquad\qquad\qquad +2\sum_{j,\alpha=1}^{m+n}(a_j,a_\alpha)\cdot\hat\phi_{j\alpha}\\
&=&0,	
\end{eqnarray*}
since $A^2=0$ and $\text{rank}\, A=1$.
\end{proof}

\begin{proposition}\label{proposition-matrices}
	Let $A$ be a complex $(m+n)\times (m+n)$ matrix such that 
	$$\Omega_{jk}(\alpha,\beta)=\det
	\begin{bmatrix}
		a_{j\alpha} & a_{j\beta}\\
		a_{k\alpha} & a_{k\beta}	
	\end{bmatrix}
	+\det
	\begin{bmatrix}
		a_{j\alpha} & a_{jk}\\
		a_{\beta\alpha} & a_{\beta k}	
	\end{bmatrix}=0,$$
	for all $1\le j,k,\alpha,\beta\le m+n$.  Then $A^2=A\cdot\trace A$ and $\rank A=1$.
\end{proposition}

\begin{proof}
	The first statement is an immediate consequence of the relation 
	$$\sum_{\alpha=1}^{m+n}\Omega_{jk}(\alpha,\alpha)=(a_j,a_k)-a_{jk}\cdot\trace A.$$
	The second statement follows from $$\Omega_{jk}(\alpha,\beta)-\Omega_{jk}(\beta,\alpha)
	=3\,\big\{\,a_{j\alpha}a_{k\beta}-\,a_{j\beta}a_{k\alpha}\big\}=3\cdot \det
	\begin{bmatrix}
		a_{j\alpha} & a_{j\beta}\\
		a_{k\alpha} & a_{k\beta}	
	\end{bmatrix}.$$
\end{proof}

In Example \ref{example-matrix-families}, we now construct matrices satisfying the conditions in Theorem \ref{theorem-real} and hence manufacture a multi-dimensional family of eigenfunctions on the real Grassmannian $\SO{m+n}/\SO m\times\SO n$.

\begin{example}\label{example-matrix-families}
Let $p=(p_1,\dots,p_{m+n})\in\cn^{m+n}$ be a non-zero isotropic element i.e. $$p_1^2+p_2^2+\cdots + p_{m+n}^2=0.$$
Then the complex $(m+n)\times (m+n)$ matrix $A=p^t\cdot p$ with $a_{jk}=p_jp_k$ satisfies the conditions $A^2=0$, $\trace A=0$ and $\rank A=1$.
Furthermore the $\SO m\times\SO n$-invariant function $\hat\Phi_p:\SO{m+n}\to\cn$ with 
$$\hat\Phi_p(x)=\sum_{j,\alpha=1}^{m+n}p_{j}p_{k}\cdot(\sum_{t=1}^{m}x_{jt}\, x_{kt}),$$ 
induces an eigenfunction $\Phi_p:\SO{m+n}/\SO m\times\SO n\to\cn$ on the quotient space with  
$$\tau(\Phi_p)=-\,(m+n)\cdot\Phi_p\ \ \text{and}\ \ \kappa(\Phi_p,\Phi_p)=-\,2\cdot\Phi_p^2.$$
This provides a complex $(m+n-1)$-dimensional family of eigenfunctions on the real Grassmannian 
$\SO{m+n}/\SO m\times\SO n$.
\end{example}

Next we explain how our construction method is inspired by the classical theory of spherical harmonics on $S^n$, as the unit sphere in the Euclidean $\rn^{n+1}$.  For this see the excellent text \cite{Axl-Bou-Ram}.

\begin{remark}\label{remark-spherical-harmonics}
For $m=1$, we identify the first column of the generic matrix element 
$$\begin{bmatrix}
	x_{11} & \cdots &x_{1,n+1}\\
	\vdots & \ddots &\vdots\\
	x_{n+1,1} & \cdots &x_{n+1,n+1}
\end{bmatrix}$$
in $\SO{n+1}$ with $x=(x_1,x_2,\dots,x_{n+1})\in\rn^{n+1}$.  Then the linear space $\H^2_{n+1}$ of second order harmonic polynomials in the coordinates of $x\in\rn^{n+1}$ is generated by the elements
$$x_1^2-x_2^2,\, x_2^2-x_3^2,\,\dots ,\,x_n^2-x_{n+1}^2,\,x_1x_2,\,x_1x_3,\,\dots ,\, x_{n-1}x_{n+1},\, x_nx_{n+1},$$
forming a basis $\B$ for $\H^2_{n+1}$.  Their restrictions to the unit sphere $S^n$ are eigenfunctions of the Laplace-Beltrami operator, all of the same eigenvalue.  

By assuming, in Theorem \ref{theorem-real}, that the matrix $A$ is traceless we see that the $\SO n$-invariant function $\hat\Phi_A:\SO{n+1}\to\cn$, given by
$$\hat\Phi_A(x)=\sum_{j,\alpha=1}^{n+1}a_{j\alpha}\cdot\hat\phi_{j\alpha}(x),$$ 
is a linear combination of the basis elements in $\B$.  If $\text{\rm rank}\,A=1$ and $\trace\,A=0$, then these functions are eigen with respect to the conformality operator $\kappa$.
\end{remark}

\section{Eigenfunctions on the Real Flag Manifolds}
\label{section-real-flag}

The standard Riemannian metric on the special orthogonal group $\SO{n}$ induces a natural metric on the real homogeneous flag manifolds
$$\F(n_1,\dots,n_t)=\SO{n}/\SO {n_1}\times\SO{n_2}\times\cdots\times\SO{n_t},$$ 
where $n=n_1+n_2+\cdots +n_t$.  For this we have the Riemannian fibrations
$$\SO{n}\to\F(n_1,\dots,n_t)\to G_{n_k}(\rn^{n}).$$ 
Let us now write the generic element $x\in\SO n$ of the form
$$x=[\,x_1\,|\,x_2\,|\,\cdots\, |\,x_t\,],$$
where each $x_k$ is an $n\times n_k$ submatrix of $x$. Following Theorem \ref{theorem-real}, we can now, for each block,  construct a family $\hat\E_k$ of $\SO {n_k}$-invariant complex-valued eigenfunctions on the special orthogonal group $\SO n$ such that for all $\hat\phi,\hat\psi\in\hat\E_k$
$$\tau (\hat\phi)=\lambda_k\cdot\hat\phi\ \ \text{and}\ \ \kappa(\hat\phi,\hat\psi)=\mu_k\cdot\hat\phi\cdot\hat\psi,$$
with $\lambda_k=-n$ and $\mu_k=-2$.
We denote by $\hat\phi_k$ the generic element in $\hat\E_k$ and then, according to Theorem \ref{theorem-p-harmonic},  each function 
$$\hat\Phi_{p,k}(x)=c_{1,k}\cdot\hat\phi_k(x)^{1-\frac{\lambda_k}{\mu_k}}\log(\hat\phi_k(x))^{p-1} + c_{2,k}\cdot\log(\hat\phi_k(x))^{p-1}$$
is proper $p$-harmonic on an open and dense subset of $\SO n$.  The sum 
$$\hat\Phi_p=\sum_{k=1}^t\hat\Phi_{p,k}$$ constitutes a multi-dimensional family $\hat\F_p$ of $\SO{n_1}\times\cdots\times\SO{n_t}$-invariant proper $p$-harmonic functions on an open dense subset of $\SO n$.  Furthermore, each element $\hat\Phi_p\in\hat\F_p$ induces a proper $p$-harmonic function $\Phi^*_p$ defined on an open and dense subset of the real flag manifold 
$$\F(n_1,\dots ,n_t)=\SO n/\SO{n_1}\times\cdots\times\SO{n_t},$$ 
which does not descend onto any of the real Grassmannians if $t\ge 3$.

\renewcommand{\arraystretch}{2}
\begin{table}[h]
	\makebox[\textwidth][c]{
		\begin{tabular}{cccc}
			\midrule
			\midrule
$U/K$	& $\lambda$ & $\mu$ & Eigenfunctions \\
\midrule
\midrule
$\SO n$ & $-\,\frac{(n-1)}2$ & $-\,\frac 12$ & see \cite{Gud-Sak-1}\\
\midrule
$\SU n$ & $-\,\frac{n^2-1}n$ & $-\,\frac{n-1}n$ & see \cite{Gud-Sob-1} \\
\midrule
$\Sp n$ & $-\,\frac{2n+1}2$ & $-\,\frac 12$ & see \cite{Gud-Mon-Rat-1} \\
\midrule
$\SU n/\SO n$ & $-\,\frac{2(n^2+n-2)}{n}$& $-\,\frac{4(n-1)}{n}$ & see \cite{Gud-Sif-Sob-2} \\
\midrule
$\Sp n/\U n$ & $-\,2(n+1)$ & $-\,2$ & see \cite{Gud-Sif-Sob-2} \\
\midrule
$\SO{2n}/\U n$ & $-\,2(n-1)$ & $-1$ & see \cite{Gud-Sif-Sob-2} \\
\midrule
$\SU{2n}/\Sp n$ & $-\,\frac{2(2n^2-n-1)}{n}$ & $-\,\frac{2(n-1)}{n}$ & see \cite{Gud-Sif-Sob-2} \\
\midrule
$\SO{m+n}/\SO m\times\SO n$ & $-(m+n)$ & $-2$ & Theorem \ref{theorem-real} \\
\midrule
\midrule
\end{tabular}	
}
\bigskip
\caption{Eigenfunctions on classical irreducible  compact Riemannian symmetric spaces.}
\label{table-eigenfunctions-2}	
\end{table}
\renewcommand{\arraystretch}{1}

\section{Acknowledgements}

The authors would like to thank Fran Burstall and Adam Lindström for useful discussions on this work.  

The first author would like to thank the Department of Mathematics at Lund University for its great hospitality during her time there as a postdoc.

\end{document}